\documentclass[12pt]{article}
\usepackage{amsmath, amsthm, amscd, amsfonts, amssymb, graphicx, color}
\usepackage[bookmarksnumbered, plainpages]{hyperref}
\usepackage[all]{xy}
\title{On generalized $\pounds-$ cotorsion LCA groups}
\author{Aliakbar Alijani}

\newtheorem{thm}{Theorem}[section]
\newtheorem{cor}[thm]{Corollary}
\newtheorem{lem}[thm]{Lemma}
\newtheorem{ex}[thm]{Example}

\newtheorem{defn}[thm]{Definition}

\numberwithin{equation}{section}

\begin{document}
\date{}
\maketitle
\begin{abstract}
A locally compact abelian group $G$ is called a generalized $\pounds-$cotosion group if $G$ contains  an open $\pounds-$cotosion subgroup $H$ such that $G/H$ is a cotorsion group. In this paper, we determine the  generalized $\pounds-$ cotorsion LCA groups.

\end{abstract}

\section{Introduction}
\newcommand{\stk}{\stackrel}
Let $\pounds$ be the category of all locally compact abelian (LCA) groups with continuous homomorphisms as morphisms. A morphism is called proper if it is open onto its image and a short exact sequence $0\to A\stackrel{\phi}{\to} B\stackrel{\psi}{\to}C\to 0$ in $\pounds$ is said to be an extension of $A$ by $C$ if $\phi$ and $\psi$ are proper morphisms. We let $Ext(C,A)$ denote the group extensions of $A$ by $C$ \cite{FG1}. A discrete group $A$ is called cotorsion if $Ext(X,A)=0$ for all discrete torsion-free groups $X$. The theory of cotorsion groups was developed by Harrison for the first time \cite{H}. For more on cotorsion groups, see \cite{F}. In \cite{Fu2}, Fulp generalized the concept of cotorsion groups to LCA groups. A group $G\in\pounds$ is called an $\pounds-$cotorsion group if $Ext(X,G)=0$ for all torsion-free groups $X\in\pounds$ \cite{Fu2}. Fulp studied the $\pounds-$cotorsion LCA groups and determined the discrete or compact $\pounds-$cotorsion groups \cite{Fu2}. In this paper, we generalize the concept of $\pounds-$ cotorsion groups. A group $G\in \pounds$ will be called a generalized $\pounds-$cotosion group if $G$ contains  an open $\pounds-$cotosion subgroup $H$ such that $G/H$ is a cotorsion group. In this paper, we determine the discrete or compact generalized $\pounds-$cotorsion groups (see Lemma \ref{6} and \ref{7}). We also determine non discrete, divisible,generalized $\pounds-$cotorsion groups (see Lemma \ref{8}).

The additive topological group of real numbers is denoted by $\Bbb R$, $\Bbb Q$ is the group of rationals with the discrete topology, $\Bbb Z$ is the group of integers and $\Bbb Z(n)$ is the cyclic group of order $n$. For any group $G$ and $H$, $tG$ is the maximal torsion subgroup of $G$ and $Hom(G,H)$, the group of all continuous homomorphisms from $G$ to $H$, endowed with the compact open topology. The dual group of $G$ is $\hat{G}=Hom(G,\Bbb R/\Bbb Z)$ and $(\hat{G},S)$ denotes the annihilator of $S\subseteq G$ in $\hat{G}$. For more on locally compact abelian groups, see \cite{HR}.

\section{Generalized $\pounds-$ cotorsion LCA groups}

\begin{defn}
A locally compact abelian group $G$ is called a generalized $\pounds-$cotosion group if $G$ contains  an open $\pounds-$cotosion subgroup $H$ such that $G/H$ is a cotorsion group.
\end{defn}

\begin{ex}
Discrete cotorsion groups and $\pounds-$cotorsion groups are generalized $\pounds-$cotosion.
\end{ex}

\begin{lem} \label{6}
A discrete group $G$ is a generalized $\pounds-$cotosion group if and only if $G$ is a cotorsion group.
\end{lem}

\begin{proof}
Let $G$ be a discrete generalized $\pounds-$cotosion group. Then, there exists an $\pounds-$ cotorsion subgroup $H$ of $G$ such that $G/H$ is a cotorsion group. By Corollary 10 of \cite{Fu2}, $H$ is a divisible torsion group. So, $0\to H\hookrightarrow G\to G/H\to 0$ splits. Hence $$G\cong H\oplus G/H$$ It follows that $G$ is a cotorsion group. Conversely is clear.
\end{proof}

\begin{lem} \label{7}
A compact group $G$ is a generalized $\pounds-$cotosion group if and only if $G\cong M\oplus \Bbb Z(n)$ which $M$ is a connected group and $n$ a positive integer number.
\end{lem}

\begin{proof}
Let $G$ be a compact generalized $\pounds-$cotosion group. Then, there exists an open $\pounds-$ cotorsion subgroup $H$ of $G$ such that $G/H$ is a cotorsion group. By Corollary 9 of \cite{Fu2}, $H$ is connected. Since $H$ is open, $H=G_{0}$. So, $0\to H\hookrightarrow G\to G/H\to 0$ splits. Hence, $G\cong H\oplus G/H$. Since $G/H$ is a compact and discrete group, $G/H$ is a finite, cotorsion group. Hence, $G/H\cong \Bbb Z(n)$ for some positive integer $n$.

conversely, let $G\cong M\oplus \Bbb Z(n)$ which $M$ is a connected group and $n$ a positive integer number. Set $H=M$. Then $H\oplus 0$ is an open $\pounds-$cotosion subgroup of $G$ such that $G/(H\oplus 0)\cong \Bbb Z(n)$ is a cotorsion group.
\end{proof}

\begin{defn}
A group $G\in \pounds$ is said to be torsion-closed if $tG$ is closed in $G$ \cite{AA}.
\end{defn}

\begin{thm} \label{1}
A group $G\in \pounds$ is torsion-closed if and only if $G\cong \Bbb R^{n}\oplus C\oplus L$ where $C$ is a compact, connected, torsion-free group and $L$ contains a compact open subgroup $H\cong \prod_{i\in I}\Bbb Z/p_{i}^{r_{i}}\Bbb Z\oplus \prod_{p}\Delta_{p}^{n_{p}}$ ($\Delta_{p}$ denotes the group of $p-$adic integers and $\Bbb Z/p_{i}^{r_{i}}\Bbb Z$ is the cyclic group of order $p_{i}^{r_{i}}$) where only finitely many distinct primes $p_{i}$ and positive integers $r_{i}$ occur and $n_{p}$ are cardinal number.\cite{AA}
\end{thm}

\begin{lem} \label{3} \label{8}
Let $G$ be a divisible group in $\pounds$. Then, $G\cong \Bbb R^{n}\oplus \hat{C}\oplus N$ where $C$ is a compact, connected, torsion-free group and $N$ contains a compact open subgroup $K$ such that $N/K$ is a discrete torsion divisible group.
\end{lem}

\begin{proof}
Let $G\in \pounds$ be a divisible group. Then, $\hat{G}$ is torsion-free and torsion-closed. By Theorem \ref{1}, $\hat{G}\cong \Bbb R^{n}\oplus C\oplus L$ where $C$ is a compact, connected, torsion-free group and $L$ contains a compact open subgroup $H\cong \prod_{i\in I}\Bbb Z/p_{i}^{r_{i}}\Bbb Z\oplus \prod_{p}\Delta_{p}^{n_{p}}$ ($\Delta_{p}$ denotes the group of $p-$adic integers and $\Bbb Z/p_{i}^{r_{i}}\Bbb Z$ is the cyclic group of order $p_{i}^{r_{i}}$) where only finitely many distinct primes $p_{i}$ and positive integers $r_{i}$ occur and $n_{p}$ are cardinal number. Set $N=\hat{L}$ and $K=(\hat{L},H)$. By Theorem 24.11 of \cite{HR}, $N/K\cong \hat{H}$. Hence, $N/K\cong B\oplus D$ where $B$ is a discrete bounded group and $D$ a discrete divisible torsion group. Since $N/K$ is divisible, $B=0$.
\end{proof}

\begin{thm} \label{2}
Let $G\in \pounds$ and $0 \to A \to B \to C \to 0$ be an extension in $\pounds$. Then, the following sequences are exact \cite{FG1}:
\begin{enumerate}
\item $0\to Hom(C,G)\to Hom(B,G)\to Hom(A,G)\to Ext(C,G)\to Ext(B,G)\to Ext(A,G)\to 0$
\item $0\to Hom(G,A)\to Hom(G,B)\to Hom(G,C)\to Ext(G,A)\to Ext(G,B)\to Ext(G,C)\to 0$
\end{enumerate}
\end{thm}

\begin{lem} \label{4}
Let $G$ be a discrete group such that $Ext(\hat{\Bbb Q},G)=0$. Then, $G$ is a torsion group.
\end{lem}

\begin{proof}
Consider the two exact sequences $0\to tG\hookrightarrow G\to G/tG\to 0$ and $0\to \hat{\Bbb Q/\Bbb Z}\to \hat{\Bbb Q}\to \hat{\Bbb Z}\to 0$. By Theorem \ref{2}, we have the two following exact sequences  $$(2.1)...\to Ext(\hat{\Bbb Q},G)\to Ext(\hat{\Bbb Q},G/tG)\to 0$$ $$(2.2)...\to Hom(\hat{\Bbb Q/\Bbb Z},G/tG)\to Ext(\hat{\Bbb Z},G/tG)\to Ext(\hat{\Bbb Q},G/tG)=0$$ By (2.1),$Ext(\hat{\Bbb Q},G/tG)=0 $. By Theorem 24.25 of \cite{HR}, $\hat{G/tG}$ is connected. Hence, by Corollary 2, p. 377 of \cite{M}, $Hom(\hat{\Bbb Q/\Bbb Z},G/tG)\cong Hom(\hat{G/tG},\Bbb Q/\Bbb Z)=0$. So, $$Hom(\hat{\Bbb Q/\Bbb Z},G/tG)=0$$ Hence, by Proposition 2.17 of \cite{FG1} and (2.2),$G/tG\cong Ext(\hat{\Bbb Z},G/tG)=0$ and $G$ is a torsion group.
\end{proof}

\begin{thm}
A non discrete, divisible group $G$ in $\pounds$ is generalized $\pounds-$cotorsion if and only if $G$ be $\pounds-$cotorsion.
\end{thm}

\begin{proof}
Let $G$ be a non discrete, divisible and generalized $\pounds-$cotorsion group. By Lemma \ref{3}, $G\cong \Bbb R^{n}\oplus N$ where $N$ contains a compact open subgroup $K$ such that $N/K$ is a discrete torsion divisible group. First, we show that $Ext(\hat{\Bbb Q},G)=0$. It is sufficient to show that $Ext(\hat{\Bbb Q},N)=0$. Consider the exact sequence $0\to K\hookrightarrow N\to N/K\to 0$. By Theorem \ref{2}, we have the following exact sequence $$...\to Ext(\hat{\Bbb Q},K)\to Ext(\hat{\Bbb Q},N)\to Ext(\hat{\Bbb Q},N/K)\to 0$$
By Theorem 2.12 and Proposition 2.17 of \cite{FG1},$Ext(\hat{\Bbb Q},K)=0$. Also, $N/K$ is an $\pounds-$cotorsion group. Hence, by Corollary 10 of \cite{Fu2},$Ext(\hat{\Bbb Q},N/K)=0$. So, $Ext(\hat{\Bbb Q},N)=0$. Hence, by Lemma \ref{4},$G/H$ is a torsion group.  Now, suppose that $H$ be an $\pounds-$ cotorsion subgroup of $G$ such that $G/H$ is a cotorsion group. Let $X$ be a torsion-free group in $\pounds$. Consider the following exact sequence $$...\to Ext(X,H)\to Ext(X,G)\to Ext(X,G/H)\to 0$$

Since $H$ is $\pounds-$cotorsion, $Ext(X,H)=0$. By Corollary 10 of \cite{Fu2},$Ext(X,G/H)=0$. Hence, $Ext(X,G)=0$ for all torsion-free groups $X\in \pounds$ and $G$ is an $\pounds-$cotorsion group. Conversely is clear.
\end{proof}

\begin{cor}
Every generalized $\pounds-$cotorsion group can be imbedded in a $\pounds-$cotorsion group.
\end{cor}

\begin{proof}
Let $G\in \pounds$ be a generalized $\pounds-$cotorsion group. Then, $G$ contains an open $\pounds-$cotosion subgroup $H$. Clearly, $H$ is an open, $\pounds-$cotorsion subgroup of $G^{*}$. Also, $G^{*}/H$ is a cotorsion group. By Theorem \ref{8}, $G^{*}$ is $\pounds-$cotorsion.
\end{proof}
\bibliographystyle{amsplain}

\bigskip
\bigskip

{\footnotesize {\bf A. A. Alijani}\; \\ {Mollasadra Technical and Vocational College}, {Technical and Vocational University,} {Ramsar, Iran.}\\
{\tt Email: alijanialiakbar@gmail.com}\\

\end{document}